\documentclass[12pt]{amsart}
\usepackage[margin=1in]{geometry}
\usepackage[utf8]{inputenc}
\usepackage{amsmath}
\usepackage{amssymb}
\usepackage{amsthm}
\usepackage{ytableau}
\usepackage{todonotes}
\usepackage{verbatim}
\usepackage{mathrsfs}
\usepackage{tikz}

\newcommand{\YY}{\mathbf Y}

\newcommand{\CC}{\mathbf C}
\newcommand{\NN}{\mathbf N}

\newcommand{\U}{\mathcal U}
\newcommand{\D}{\mathcal D}
\newcommand{\B}{\mathcal B}
\newcommand{\C}{\mathcal C}

\newcommand{\s}[1]{\mathsf{#1}}

\title{Up- and Down-Operators on Young's Lattice}
\author{Ricky Ini Liu}
\address{Department of Mathematics, North Carolina State University, Raleigh, NC}
\email{riliu@ncsu.edu}
\author{Christian Smith}
\address{Department of Mathematics, North Carolina State University, Raleigh, NC}
\email{casmit34@ncsu.edu}
\date{\today}
\thanks{\emph{Keywords}: Schur operator, plactic monoid, Young's lattice, up- and down-operators.}
\thanks{R. I. Liu and C. Smith were partially supported by National Science Foundation grant DMS-1700302. R.~I.~Liu was also partially supported by National Science Foundation grant CCF-1900460.}

\newtheorem{theorem}{Theorem}[section]
\newtheorem*{theorem*}{Theorem}
\newtheorem{corollary}[theorem]{Corollary}
\newtheorem{proposition}[theorem]{Proposition}
\newtheorem{lemma}[theorem]{Lemma}
\theoremstyle{definition}
\newtheorem{definition}[theorem]{Definition}
\newtheorem{example}[theorem]{Example}

\newtheorem{case}{Case}

\numberwithin{subcase}{case}

\begin{document}

\begin{abstract}
The up-operators $u_i$ and down-operators $d_i$ (introduced as Schur operators by Fomin) act on partitions by adding/removing a box to/from the $i$th column if possible. It is well known that the $u_i$ alone satisfy the relations of the (local) plactic monoid, and the present authors recently showed that relations of degree at most 4 suffice to describe all relations between the up-operators. Here we characterize the algebra generated by the up- and down-operators together, showing that it can be presented using only quadratic relations.
\end{abstract}

\maketitle

\section{Introduction}
The \emph{up-operators} $u_i$ for $i \in \NN$ act on a partition $\lambda$ by adding a box to the $i$th column of $\lambda$ if the result is a partition and by sending $\lambda$ to $0$ otherwise.  Similarly, the \emph{down-operators} $d_i$ act on $\lambda$ by subtracting a box from the $i$th column if the result is a partition and by sending it to $0$ otherwise.  These operators were introduced as \emph{Schur operators} by Fomin~\cite{fomin} and further discussed by Fomin and Greene~\cite{fomingreene} in the context of noncommutative Schur functions. They can also be seen as refinements of the raising and lowering operators $U$ and $D$ acting on Young's lattice as defined by Stanley \cite{stanley} in his study of differential posets.

It was noted in \cite{fomingreene} that the $u_i$ give a representation for the \emph{local plactic monoid} as they satisfy the relations:
\begin{alignat*}{2}
    u_{i} u_{j} &= u_{j}  u_{i} &\qquad \text{for } |i - j|  &\geq 2, \\
    u_i  u_{i +  1} u_i &=  u_{i  +  1}  u_i  u_i,  \\
    u_{i  +  1 }   u_i   u_{i  +  1}  &=   u_{i  +  1}    u_{i  +  1}   u_i. 
\end{alignat*}
(In particular, the $u_i$ satisfy the classical Knuth relations of the plactic monoid---see for instance \cite{LascouxSchutzenberger}.)
The current authors proved in \cite{LS} (see also Meinel \cite{Meinel:2019dgk}) that the $u_i$ also satisfy the additional degree 4 relation 
\[u_{i  +  1}    u_{i  +  2}   u_{i  +  1}  u_i =  u_{i  +  1}   u_{i  +  2 }   u_i u_{i  +  1}\]
and that this relation along with the local plactic relations characterize the algebra generated by the $u_i$, therein called the \emph{algebra of Schur operators}.  

It was also noted in \cite{fomin} (using the fact that the down-operators can be thought of as transposes of the up-operators) that the $d_i$ satisfy:
\begin{alignat*}{2}
    d_{j} d_{i} &= d_{i}  d_{j} & \qquad \text{for } |i - j| &\geq 2,   \\
    d_{i}  d_{i+1} d_{i} &=  d_{i}  d_{i}  d_{i+1}, & \\
    d_{i+1}   d_{i}   d_{i+1}  &=   d_{i}    d_{i+1}   d_{i+1}, &
\end{alignat*}
and that together the $u_i$ and $d_i$ satisfy:
\begin{alignat*}{2}
    d_i  u_j &= u_j d_i  & \qquad \text{for } i \neq j,   \\
    d_1  u_1 &= id,  & \\
    d_{i+1} u_{i + 1}  &= u_i d_i. & 
\end{alignat*}

In this paper we give a complete description of the algebra generated by the $u_i$ and $d_i$, which we call the \emph{algebra of up- and down-operators for Young's lattice}. Surprisingly, the following theorem shows that quadratic relations suffice to give a presentation of this algebra.
\begin{theorem}
\label{maintheorem}
The algebra of up- and down-operators for Young's lattice is defined by the relations:
\begin{alignat*}{2}
     u_i  u_j &=  u_j  u_i  &&\qquad \text{for } | i-  j| \geq 2,  \\
    d_i d_j &= d_j d_i  && \qquad \text{for }|i-j| \geq 2,  \\
    d_i  u_j &= u_j d_i  &&  \qquad \text{for } i \neq j,   \\
    d_1  u_1 &= id,  & \label{5}\\
    d_{i+1} u_{i + 1} &= u_i d_i. & 
\end{alignat*}
\end{theorem}
In particular, it follows that the local plactic relations are implied by the quadratic relations in Theorem \ref{maintheorem}.

In contrast, we also give a complete description of the subalgebra generated by $u_t$ and $d_t$ for a fixed $t>1$ and show that it cannot be presented using relations of bounded degree.  

We provide necessary background information about partitions and the up- and down-operators in Section 2.  The characterization of the algebra of up- and down- operators is given in Section 3, and a discussion of subalgebras can be found in Section 4.

\section{Preliminaries}

In this section, we discuss some background on partitions and up- and down-operators.
\subsection{Partitions}
A \emph{partition} $\lambda = (\lambda_1, \lambda_2, \ldots)$ of $|\lambda| =  \sum_{i=1} \lambda_i$ is a sequence of nonincreasing nonnegative integers.  We associate to each partition a collection of left-aligned boxes with $\lambda_i$ boxes in the $i$th row called the \emph{Young diagram} of $\lambda$.  We define the \emph{conjugate partition} $\lambda'$ to be the partition whose Young diagram is obtained by reflecting the Young diagram of $\lambda$ across the main diagonal.

We consider the partial order on partitions $\lambda$ and $\mu$ such that $\mu \leq \lambda$ if and only if the Young diagram of $\mu$ fits inside the Young diagram of $\lambda$, that is, $\mu_i \leq \lambda_i$ for all $i$.  Note that this means that if $\mu \leq \lambda$, then $\lambda$ covers $\mu$ (denoted $\mu \lessdot \lambda$) if and only if $\lambda / \mu$ is a single box, where $\lambda / \mu$ is the \emph{skew Young diagram} consisting of all boxes in $\lambda$ that are not in $\mu$. We take \emph{Young's lattice} $(\YY, \leq)$ to be the partially ordered set of partitions with the above partial order. 

\subsection{Words in the alphabet}
Let $\NN = \{ 1, 2, \ldots \}$, $\overline{\NN} = \{ \overline{1}, \overline{2}, \ldots \}$, and $\Gamma = \NN \cup \overline{\NN}$.  We refer to elements $1, 2, \ldots$ of $\NN$ as unbarred letters and elements $\overline{1}, \overline{2}, \ldots$ of $\overline{\NN}$ as barred letters.

Let $x = x_1 \cdots x_{\ell}$ be a word of \emph{length} $\ell$ in the alphabet $\Gamma$.  The \emph{weight} of $x$ is the vector $w(x) = (w_1(x), w_2(x), \ldots )$ where 
\[
w_i(x) = (\text{the number of times $i$ appears in $x$}) -  (\text{the number of times $\overline{i}$ appears in $x$}).
\]
We also define the $\alpha$-\emph{vector} of $x$ to be $\alpha(x) = (\alpha_1(x), \alpha_2(x), \ldots)$ where
\[
\alpha_i(x) = \max \{ w_{i+1}(\tilde{x}) - w_i(\tilde{x}) \mid \tilde x \text{ is a suffix subword of } x  \}.
\]  Here a suffix subword $\tilde{x}$ is a word of the form $\tilde{x} = x_j x_{j+1} \cdots x_{\ell}$ for some $1 \leq j \leq \ell + 1$. When $j = \ell +1$, $\tilde{x}$ is the empty word, in which case $w_{i+1}(\tilde{x}) = w_i(\tilde{x}) = 0$, so it follows that $\alpha_i(x) \geq 0$ for all $i$.

\begin{example}
Let $x = 1 1 \overline{3} \overline{3}  \overline{2}  3 2  \overline{1} 2 1 $.  Then $w(x) = (2, 1, -1, 0, \ldots)$ and $\alpha(x) = (2, 0, 1, 0, \ldots)$. For instance, for $\alpha_1(x)=2$, the maximum value of $w_2(\tilde x)-w_1(\tilde x)$ first occurs when $\tilde x = 2\overline{1}21$.
\end{example}

\subsection{Up-operators and down-operators} 
Let $\U$ be the free associative algebra over the complex field $\CC$ generated by elements $u_{i}$ for $i \in \Gamma$. We will write $d_i = u_{\overline{i}}$ for all barred letters $\overline{i}$.  For any word $x = x_1 x_2 \ldots x_{\ell}$ in the alphabet $\Gamma$, we define $u_x = u_{x_1}u_{x_2} \cdots u_{x_{\ell}}$.
We also use the alternate notation  $\s i = u_i$  for $i \in \Gamma$.  To avoid potential confusion in the future, we note now that $\s{( i   +  j  )}$ denotes  $u_{i+j}$ and \emph{not} the sum $u_i + u_j$. 

Let $\CC[\YY]$ be the complex vector space with basis $\YY$.  We define an action of $\U$ on $\CC[\YY]$ in the following way.  For $\lambda \in \YY$ and $i \in \NN$, we let
\[
u_i (\lambda) =
\begin{cases}
\mu & \text{if } \mu \in \YY \text { and } \mu/ \lambda \text{ is a single box in column $i$},\\
0 & \text{otherwise,}
\end{cases}
\]
and
\[
d_i (\lambda) =
\begin{cases}
\mu & \text{ if } \mu \in \YY \text { and }  \lambda / \mu \text{ is a single box in column $i$},\\
0 & \text{otherwise}.
\end{cases}
\]

\begin{example}
	Let $\lambda = (3,1)$. Then $u_2(\lambda) = (3,2)$, $d_3 u_2(\lambda) = (2,2)$, but $d_1d_3u_2(\lambda) = 0$ since subtracting a box from the first column does not yield a partition.
	\ytableausetup{smalltableaux}
	\[\ydiagram{3,1} \;\xrightarrow{u_2}\; \ydiagram{3,2} \;\xrightarrow{d_3}\; \ydiagram{2,2} \;\xrightarrow{d_1}\; 0\]
\end{example}

Note that $u_i(\lambda)$ is either $0$ or a partition that covers $\lambda$ in $\YY$, so we refer to $u_i$ as an \emph{up-operator}, and similarly we call $d_i$ a \emph{down-operator}. These operators were introduced by Fomin \cite{fomin} under the name \emph{Schur operators}.

%\begin{example}
%Let $\lambda = (3,3,1)$.  Then $c_{\overline{1}}(\lambda) = (3,3), c_{4}c_{\overline{1}}(\lambda) = (4,3),$ and $c_{\overline{2}} c_{4} c_{\overline{1}}(\lambda) = 0$.
%\end{example}

The action of $u_x$ on partitions is determined by the weight and $\alpha$-vector of $x$ as follows.

\begin{proposition} \label{operatoraction}
	Let $x$ be a word and $\lambda \in \YY$.  Then
	\[
	u_x (\lambda) =
	\begin{cases}
		(\lambda'_1 + w_1(x), \lambda'_2 + w_2(x), \ldots)' & \text{ if } \lambda'_i - \lambda'_{i+1} \geq \alpha_i(x) \text{ for all } i, \\
		0  & \text{ otherwise}.
	\end{cases}
	\]
\end{proposition}
\begin{proof}
	We have $u_x(\lambda) \neq 0$ if and only if $u_{\tilde{x}}(\lambda)$ is a partition for each suffix subword $\tilde{x}$ of $x$.  Fix some $\tilde{x}$ and suppose $\mu = u_{\tilde{x}}(\lambda) \neq 0$.  We then have $\mu_i' = \lambda_i' + w_i(\tilde{x})$ for all $i$.  The condition for $\mu$ to be a partition is that $\mu_i' \geq \mu_{i+1}'$ for all $i$, or equivalently 
	\[
	\lambda_i' + w_i(\tilde{x}) \geq \lambda_{i+1}' + w_{i+1}(\tilde{x}).
	\]
	Rearranging this gives 
	\[
	\lambda_i' - \lambda_{i+1}'  \geq  w_{i+1}(\tilde{x}) - w_i(\tilde{x}). 
	\]
	By the definition of $\alpha_i(x)$, these inequalities hold for all suffix subwords $\tilde{x}$ if and only if $\lambda_i' - \lambda_{i+1}'  \geq \alpha_i(x)$.  
\end{proof}

The following corollary then follows from Proposition~\ref{operatoraction}.
\begin{corollary} \label{whensame}
	Let $x$ and $y$ be words.  Then $u_x$ and $u_y$ act identically on $\YY$ if and only if $\alpha(x) = \alpha(y)$ and $w(x) = w(y)$.
\end{corollary}
\begin{proof}
	The backwards implication is immediate from Proposition~\ref{operatoraction}. For the forward direction, suppose $\alpha(x) \neq \alpha(y)$. Then we may assume without loss of generality that $\alpha_j(x) < \alpha_j(y)$ for some $j$. Taking $\lambda$ such that $\lambda'_i-\lambda'_{i+1} = \alpha_i(x)$, we have $u_x(\lambda) \neq 0 = u_y(\lambda)$, so $u_x$ and $u_y$ do not act identically. If instead $\alpha(x) = \alpha(y)$ but $w(x) \neq w(y)$, then for this same choice of $\lambda$, $u_x(\lambda) \neq u_y(\lambda)$ by Proposition~\ref{operatoraction}.
\end{proof}

 It was noted in \cite{fomin} that $u_i$ and $d_i$ are transposes with respect to the basis $\YY$, which we may write as $u^T_i = d_i$. Also in \cite{fomin}, various relations among the $u_i$ and $d_i$ were described, including the local plactic relations and various quadratic relations (as described in Section 3). Our main result will be to show that in fact these quadratic relations generate all relations between these operators.

\section{The algebra of up- and down-operators} \label{secupanddown}
Let $I$ be the two-sided ideal consisting of all elements of $\U$ that annihilate $\CC[\YY]$.  We call the algebra $\U/I$ the \emph{algebra of up- and down-operators for Young's lattice}.  Let $J$ be the two-sided ideal generated by the following relations.
\begin{alignat}{2}
     u_i  u_j &\equiv  u_j  u_i  &&\qquad\text{ for }| i-  j| \geq 2,  \label{1}\\
    d_i d_j &\equiv d_j d_i  &&\qquad\text{ for }|i-j| \geq 2,  \label{3}\\
    d_i  u_j &\equiv u_j d_i  &&\qquad\text{ for } i \neq j,   \label{6}\\
    d_1  u_1 &\equiv id,  && \label{5}\\
    d_{i+1} u_{i + 1} &\equiv u_i d_i. &&  \label{7}
\end{alignat}
Our main result (Theorem~\ref{maintheorem}) will be to show that $I = J$. We first verify that $J \subseteq I$.

\begin{proposition} \label{JsubsetI}
The inclusion of ideals $J \subseteq I$ holds.
\end{proposition}
\begin{proof}
It suffices to show that for each of (\ref{1})--(\ref{7}), the two terms in the relation are in fact equivalent modulo $I$.  We show this for relation (\ref{7}); the other relations are similar. By Corollary \ref{whensame} we need only show $\alpha(x) = \alpha(y) \text{ and }  w(x) = w(y)$ where $x = i \overline{i}$ and $y = \overline{(i+1)}(i+1)$.  Indeed, $w(x) = (0,0, \dots) = w(y)$, while $\alpha(x) = (0, \ldots, 0, 1, 0, \ldots) = \alpha(y)$, where the $1$ occurs in the $i$th position.
\end{proof}
It therefore remains only to show that $I \subseteq J$. The next proposition proves that $I$ is a \emph{binomial ideal}, that is, $I$ is generated by elements of the form $u_x-u_y$.  %This fact will make proving $I \subseteq J$ simpler since it means that we only need to show that $u_x \equiv u_y \pmod{I}$ implies $u_x \equiv u_y \pmod{J}$ to get our desired inclusion.
The proof of this proposition is very similar to that of Proposition~3.3 in \cite{LS}, but we include it here for completeness.
\begin{proposition} \label{binomialideal}
The ideal $I$ is a binomial ideal.
\end{proposition}
\begin{proof}
 Let $I'$ be the two-sided ideal generated by all binomials $u_x - u_y$ such that $u_x \equiv u_y \pmod I$, and suppose $R \in I$.  Since $\U$ is graded by weight and $I$ is homogeneous with respect to weight, we may assume that all terms appearing in $R$ have weight $w$ for some $w = (w_1, w_2, \ldots)$.   We can then find $R' \equiv R \pmod{I'}$ for some
\begin{align*}
R' = \sum_{i=1}^n c_{x^{(i)}} u_{x^{(i)}},
\end{align*}
where $x^{(i)}$ is a word in $\Gamma$ of weight $w$, $u_{x^{(i)}} \not \equiv u_{x^{(j)}} \pmod{I}$ whenever $i \neq j$, and $0 \neq c_{x^{(i)}} \in \CC$ for all $i \in [n]$.  In particular, by Corollary~\ref{whensame}, the $\alpha(x^{(i)})$ are distinct, so suppose without loss of generality that they occur in lexicographic order.  %Note that we have strict inequality in the lexicographical ordering since otherwise Corollary \ref{whensame} would imply that there exists some $i$ such that $u_{x^{(i)}} \equiv u_{x^{(i+1)}} \pmod{I}$, which is a contradiction.   

If $n \geq 1$, let $\lambda \in \YY$ be such that $\lambda_k' - \lambda_{k+1}' = \alpha_k(x^{(1)})$ for all $k$.  By Proposition~\ref{operatoraction}, $u_{x^{(1)}}(\lambda) \neq 0$. For each $i > 1$, by the lexicographic ordering, there exists some $s$ such that $\alpha_s(x^{(1)}) < \alpha_s(x^{(i)})$. Then by Proposition~\ref{operatoraction}, $u_{x^{(i)}}(\lambda) = 0$. Thus $0 = R'(\lambda) = c_{x^{(1)}}u_{x^{(1)}}(\lambda)$, which implies $c_{x^{(1)}} = 0$. This is a contradiction, so we must have $R'=0$. Thus $I=I'$.
\end{proof}

Our goal for the rest of this section is to show that if $u_x \equiv u_y \pmod{I}$, then $u_x \equiv u_y \pmod{J}$.  Our general strategy is as follows.  Let $[u_x]_I$ be the equivalence class of $u_x$ modulo $I$.  %For sufficiently large $m,n \in \NN$, we construct a word $[x]_{m,n}$ such that that $u_{[x]_{m,n}} \in [u_x]_I$.  The word $[x]_{m,n}$ will be dependent only upon the vectors $\alpha(x)$ and $w(x)$, and so by Corollary \ref{whensame}, $[x]_{m,n} = [y]_{m,n}$. % and $u_{[x]_{m,n}} = u_{[y]_{m,n}}$. 
We will construct a representative word $[x]$ such that $u_{[x]} \in [u_x]_I$. This representative will only depend on $\alpha(x)$ and $w(x)$, so if $u_x \equiv u_y \pmod{I}$, then $[x]=[y]$.
We will then show that $u_x \equiv u_{[x]} \pmod{J}$ and similarly for $y$, which will complete the proof. 
\begin{definition}
For a word $x$, define 
\begin{align*}
m(x) &= \max_{i \in \NN} \{-(\alpha_i(x) + w_i(x))\} \geq 0, \\
n(x) &= \max \{t \in \NN \mid t \text{ or } \overline{t} \text{ appears in } x\}.
\end{align*}
For any $m \geq m(x), n \geq n(x)$, we let
\[
[x]_{m,n} = (\overline{ 1}^{m} \cdots \overline{ n}^{m}) (  n^{\beta^m_n(x)} \overline{ n}^{\alpha_n(x)} \cdots   1^{\beta^m_1(x)} \overline{ 1}^{\alpha_1(x)})
%(\overline{\s 1}^{m} \cdots \overline{\s n}^{m}) ( \s n^{\beta^m_n(x)} \overline{\s n}^{\alpha_n(x)} \cdots  \s 1^{\beta^m_1(x)} \overline{\s 1}^{\alpha_1(x)})
\]
where $\beta^m_i(x) = \alpha_i(x) + w_i(x) + m$.
\end{definition}
%When $m = m(x)$ and $n = n(x)$ we let $[x] = [x]_{m,n}$.  The operator $u_{[x]}$ is the standard equivalence class representative mentioned earlier. 
%For some values of $m,n$ we take $u_{[x]_{m,n}}$ to be the equivalence class representative mentioned earlier.  We will see later how to determine the specific values of $m$ and $n$. 
Note that the definition of $m$ ensures that all of the exponents appearing in the definition of $[x]_{m,n}$ are nonnegative. We will often abbreviate $[x] = [x]_{m,n}$.
%  This makes computations involving $u_{[x]_{m,n}}$ more convenient since operators with negative exponents are undefined. 
We now show that indeed $u_{[x]} \in [u_x]_I$.

\begin{proposition}
For any word $x$, $u_x \equiv  u_{[x]}  \pmod{I}$.
\end{proposition}
\begin{proof}
Let $i \in \NN$. Then $w_i([x]) = -m+\beta_i^m(x)-\alpha_i(x) = w_i(x)$.  We now  show that $\alpha_i([x]) = \alpha_i(x)$.  For ease of notation, we will write $\alpha_i = \alpha_i(x), w_i = w_i(x),$ and $\beta_i = \beta^m_i(x)$.  Since $\alpha_i$ relies only upon the appearances of $i$, $\overline{i}$, $(i+1)$, and $\overline{(i+1)}$ in $x$, we need only consider the subword 
\[
\overline{i}^m \overline{(i+1)}^m (i+1)^{\beta_{i+1}} \overline{(i+1)}^{\alpha_{i+1}} i^{\beta_i} \overline{i}^{\alpha_i}.
\]
To calculate $\alpha_i([x])$, we need to find the maximum value of $w_{i+1}(\tilde x)-w_i(\tilde x)$ for each suffix subword $\tilde x$. This value only increases when adding an occurrence of $\overline i$ or $(i+1)$ to $\tilde x$. Thus we need only verify a few choices of $\tilde x$:
\begin{alignat*}{2}
	\tilde x &= \overline{i}^{\alpha_i}:\qquad& w_{i+1}(\tilde x)-w_i(\tilde x) &= \alpha_i,\\
	\tilde x &= (i+1)^{\beta_{i+1}} \overline{(i+1)}^{\alpha_{i+1}} i^{\beta_i} \overline{i}^{\alpha_i}:\qquad& w_{i+1}(\tilde x)-w_i(\tilde x) &= \beta_{i+1}-\alpha_{i+1}-\beta_i+\alpha_i\\&&& = w_{i+1}-w_i,\\
	\tilde x &= [x]:\qquad &w_{i+1}(\tilde x)-w_i(\tilde x) &= w_{i+1}-w_i.
\end{alignat*}
The maximum of these is just $\alpha_i$.
%It is straightforward to check that $\alpha_i([x]_{m,n}) = \max \{\alpha_i, w_{i+1} - w_i \} = \alpha_i$ where $w_{i+1}(
%\tilde{x}) - w_i(\tilde{x}) = \alpha_i  $ when $\tilde{x} = \overline{i}^{\alpha_i}$ and $w_{i+1}(\tilde{x}) - w_i(\tilde{x}) = w_{i+1} - w_i$ when $\tilde{x} = (i+1)^{\beta_{i+1}} \overline{(i+1)}^{\alpha_{i+1}} i^{\beta_i} \overline{i}^{\alpha_i}$. 
%
%
%If $\tilde{x} = \overline{i}^{\alpha_i}$, then %$w_{i+1}(\tilde{x}) - w_i(\tilde{x}) = %\alpha_i$.  On the other hand, if $\tilde{x} = % (i+1)^{\beta_{i+1}} %\overline{(i+1)}^{\alpha_{i+1}} i^{\beta_i} %\overline{i}^{\alpha_i}$, then %$w_{i+1}(\tilde{x}) - w_i(\tilde{x}) = %\alpha_i - \beta_i - \alpha_{i+1} + %\beta_{i+1} = w_{i+1} - w_i$.  The previous %case implies that $w_{i+1} - w_i \geq %\alpha_i$.  Strict inequality contradicts the %definition of $\alpha_i$ so we must have  %$w_{i+1}(\tilde{x}) - w_i(\tilde{x}) = %\alpha_i$.  
%
%
%This value relies solely upon the appearances of %$i, \overline{i}, (i+1), \overline{(i+1)}$ in %$[x]_{m,n}$.  From this we see that %$\alpha_i([x]_{m,n}) = \alpha_i(z)$ where $z$ is %the subword of $[x]_{m,n}$ consisting of these %four letters.  For ease of notation we let %$\alpha_i = \alpha_i(x), \beta_i = \beta_i^m(x),$ %and $j = (i+1)$.  We have
%\[
%z = \overline{i}^m \overline{j}^m j^{\beta_{j}} %\overline{j}^{\alpha_{j}} i^{\beta_i} %\overline{i}^{\alpha_i}.
%\]
%We see that $\alpha_i(z) = \alpha_i$ and so  %$\alpha_i([x]_{m,n}) = \alpha_i$.  By Proposition %\ref{whensame} we have $u_x \equiv  u_{[x]_{m,n}} % \pmod{I}$.
\end{proof}
We now wish to show that $u_x \equiv u_{[x]_{m,n}} \pmod{J}$ for sufficiently large $m$ and $n$.  To that end we will make use of the following two lemmas. As a reminder, we will use $\s{i}$ and $\s{\overline{i}}$ to represent $u_i$ and $d_i$, respectively.

%The next several lemmas are needed for an intermediate step in this task; namely, proving that $u_{[x]_{m,n}} \equiv u_{[x]} \pmod{J}$.

\begin{lemma} \label{newidentity}
Let $x = \overline{1} \cdots \overline{n} n \cdots 1$ for any $n \in \NN$.  Then $u_x \equiv id \pmod{J}$.
\end{lemma}
\begin{proof}
First note the equivalence
\begin{align}
\s{\overline{n} n (n-1) \cdots 1} &\equiv \s{(n-1)(n-2) \cdots 1 }  \pmod{J}, \label{specialidentity}    
\end{align}
which holds by repeated application of (\ref{7}) and a single use of (\ref{5}).
% for any $n \in \NN$.  Suppose $n=1$.  We then have %$\s{\overline{1}1} \equiv \s{\overline{1}1} \pmod{J}$ %which is trivially true.  Now suppose the statement %holds for all $n \leq s-1$ for some $s-1 \in \NN$ and %let $n = s$.  We have 
%\begin{align*}
%\s{(n-1)(n-2)(n-3) \cdots 1 \overline{1} 1} &\equiv %\s{(n-1) \overline{(n-1)} (n-1) (n-2) \cdots 1} & %\text{(induction)} \\
% &\equiv \s{\overline{n}n  (n-1) \cdots 1} & %(\ref{7}).
%\end{align*}
%To prove the lemma, we induct on $n$.  If $n=1$, then the lemma statement holds by (\ref{5}).  Suppose the equivalence holds for all $n \leq s-1$ for some $s-1 \in \NN$ and let $n = s$.  We have  
Then
\begin{align*}
\s{\overline{1} \cdots \overline{n} n \cdots 1}
&\equiv \s{\overline{1} \cdots \overline{(n-1)} (n-1) \cdots 1}\\
&\equiv \s{\overline{1} \cdots \overline{(n-2)} (n-2) \cdots 1}\\
&\;\;\vdots\\
&\equiv id
\end{align*}
by repeated application of \eqref{specialidentity}.
\end{proof}

\begin{lemma} \label{discardedrelations}
The following equivalences hold modulo $J$:
\begin{align}
u_i &\equiv u_i d_i u_i,  &  \label{8}\\ 
d_i &\equiv  d_i u_i d_i,   &\label{9} \\
 u_i u_{i  +  1}  u_i &\equiv u_{i +  1} u_i u_i, &  \label{2}\\
d_i  d_{i+1} d_i &\equiv d_i d_i d_{i+1}, &  \label{4} \\
u_{i+1} u_i u_{i+1} &\equiv u_{i+1} u_{i+1} u_i, \label{10} \\
d_{i+1} d_i d_{i+1} &\equiv d_i d_{i+1} d_{i+1}. \label{11} 
\end{align}
\end{lemma}
\begin{proof}
For (\ref{8}), we have
\begin{alignat*}{2}
    \s{n} &\equiv \s{n \overline{1} \cdots \overline{n} n \cdots 1} &&(\text{Lemma } \ref{newidentity}) \\
    &\equiv \s{ \overline{1} \cdots \overline{(n-1)}n\overline{n} n \cdots 1} &&(\ref{6}) \\
    &\equiv \s{ \overline{1} \cdots \overline{(n-1)} \, \overline{(n+1)} (n+1) n \cdots 1} &&(\ref{7}) \\
    &\equiv \s{\overline{(n+1)} (n+1) n \overline{1} \cdots \overline{(n-1)}(n-1)  \cdots 1} \qquad\qquad &&(\ref{3}), (\ref{6}) \\
    &\equiv \s{\overline{(n+1)} (n+1) n} &&(\text{Lemma }\ref{newidentity}) \\
    &\equiv \s{n \overline{n} n}. &&(\ref{7})
\end{alignat*}
For (\ref{2}) we have
\begin{alignat*}{2}
    \s{n (n+1) n} &\equiv \s{n (n+1) \overline{(n+1)} (n+1) n} \qquad\qquad&&(\ref{8}) \\
    &\equiv \s{n  \overline{(n+2)} (n+2) (n+1) n} &&(\ref{7})  \\
    &\equiv \s{  \overline{(n+2)} (n+2) n (n+1) n} &&(\ref{1}), (\ref{6}) \\
    &\equiv \s{ (n+1) \overline{(n+1)}  n (n+1) n} &&(\ref{7}) \\
    &\equiv \s{ (n+1)  n  \overline{(n+1)} (n+1) n} &&(\ref{6}) \\
    &\equiv \s{ (n+1)  n n \overline{n}  n} &&(\ref{7}) \\
    &\equiv \s{ (n+1)  n n}. &&(\ref{8})
\end{alignat*}
For (\ref{10}), assume
%$\s{221} \stackrel{(\ref{8})}{\equiv} \s{2 %\overline{2} 2 21} \stackrel{(\ref{7})}{\equiv} \s{2 %1 \overline{1} 21} \stackrel{(\ref{6})}{\equiv}  \s{2 %1 2 \overline{1} 1} \stackrel{(\ref{5})}{\equiv} %\s{212}$. 
$n \geq 2$ (the case when $n=1$ is similar). Then
\begin{alignat*}{2}
    \s{(n+1)n(n+1)} &\equiv \s{(n+1)n \overline{n} n (n+1)} &&(\ref{8}) \\
    &\equiv  \s{(n+1)n (n-1)\overline{(n-1)}  (n+1)} \qquad\qquad&&(\ref{7}) \\
    &\equiv  \s{(n+1)n  (n+1) (n-1)\overline{(n-1)}}  &&(\ref{1}), (\ref{6}) \\
    &\equiv  \s{(n+1)n  (n+1) \overline{n} n}  &&(\ref{7}) \\
    &\equiv  \s{(n+1)n \overline{n}  (n+1)  n}  &&(\ref{6}) \\
    &\equiv  \s{(n+1) \overline{(n+1)} (n+1)  (n+1)  n}  &&(\ref{7}) \\
    &\equiv  \s{(n+1)  (n+1)  n}.  &&(\ref{8})
\end{alignat*}
The proofs of  (\ref{9}), (\ref{4}), and (\ref{11})  are similar to the proofs of  (\ref{8}), (\ref{2}), and (\ref{10}), respectively.  
\end{proof}

In particular, one can observe that \eqref{2} and \eqref{10} are \emph{Knuth relations}, which, together with \eqref{1}, verify that the quadratic relations imply that $J$ contains all the relations of the \emph{local plactic monoid} generated by the $u_i$ (see \cite{fomin}).

We are now ready to prove the heart of our main theorem.

\begin{proposition}
\label{u_x=u_[x]}
Let $x = x_1 \cdots x_{\ell}$ be a word.  Then there exist $M,N \in \NN$ such that $u_x \equiv u_{[x]_{m,n}}  \pmod{J}$ for all $m \geq M$, $n \geq N$.
\end{proposition}

\begin{proof}
 As before, we will abbreviate $[x] = [x]_{m,n}$ and $[y] = [y]_{m,n}$.  We proceed by induction on the length of $x$.  First suppose $\ell = 0$ (that is, $x$ is the empty word), and take any $m,n \geq 0$.   Then we have $[x] = \overline{1}^m \ldots \overline{n}^m n^m \ldots 1^m$ and we wish to show that $u_{[x]} \equiv id \pmod{J}$.  %If $m = 0$, then the statement is trivially true, while if $m = 1$, then the statement is true by Lemma~\ref{newidentity}.  Suppose then that $m> 1$.  
% We will show that for $m \geq 1$,
%\begin{align}
%\overline{ \s 1}^{m} \cdots \overline{ \s n}^{m} \s n^{m} \cdots \s 1^{m} \equiv \overline{ \s 1}^{m-1} \cdots \overline{ \s n}^{m-1} \s n^{m-1} \cdots \s 1^{m-1} \pmod{J}. \label{biggernewidentity}
%\end{align}
%Repeated application of (\ref{biggernewidentity}) will then give $u_{[x]} \equiv id \pmod{J}$. By Lemma~\ref{newidentity},
%We will first prove the following equivalences modulo $J$:
%\begin{align}
%    \s{n \cdots 1 } \s n^{m_n} \cdots \s1^{m_1} %&\equiv  \s n^{m_n+1} \cdots \s1^{m_1+1}, %\label{right} \\ 
%    \overline{\s1}^{m_1} \cdots \overline{\s %n}^{m_n}  \s{ \overline{1}  \cdots \overline{n}} %&\equiv   \overline{\s1}^{m_1+1}  \cdots %\overline{\s n}^{m_n+1}. \label{left}
%\end{align}
%We prove (\ref{right}) by induction on $n$.  If %%$n=1$ the statement is clearly true.  Suppose %the statement holds for all $n \leq s-1 \in \NN$ %and let $n=s$.  Then,
%\begin{align*}
%\s{n \cdots 1 } \s n^{m_n}  \cdots \s1^{m_1} %&\equiv \s{n (n-1) } \s n^{m_n} \s{(n-2) \cdots %1} \s{(n-1)}^{m_{n-1}} \cdots \s1^{m_1} %&(\ref{1})  \\ 
%&\equiv \s n^{m_n+1} \s{(n-1) \cdots 1} %\s{(n-1)}^{m_{n-1}} \cdots \s1^{m_1} &(\ref{10}) %\\
%&\equiv \s n^{m_n+1}  \s{(n-1)}^{m_{n-1}+1} %\cdots \s1^{m_1+1} &(\text{induction}).
%\end{align*}
%The proof of (\ref{left}) is similar.  Finally, %we have
%\[\overline{ \s 1}^{m-1} \cdots \overline{ \s n}^{m-1} \s n^{m-1} \cdots \s 1^{m-1} \equiv \overline{ \s 1}^{m-1} \cdots \overline{ \s n}^{m-1} \s{\overline{1} \cdots \overline{n} n \cdots 1} \s n^{m-1} \cdots \s 1^{m-1} \pmod{J}.\]
By \eqref{1} and \eqref{10},
\[\s{n\cdots 1}\s{n} \equiv \s{n}\s{(n-1)}\s{n}\s{(n-2)\cdots 1} \equiv \s{nn(n-1)\cdots 1}.\]
In other words, $\s{n}$ and $\s{n\cdots 1}$ commute. Therefore
\begin{align*}
	(\s{n \cdots 1})\s{n}^{m-1} \cdots \s{1}^{m-1} &\equiv \s{n}^m (\s{(n-1) \cdots 1})\s{(n-1)}^{m-1} \cdots \s{1}^{m-1}\\ &\equiv \s{n}^m \s{(n-1)}^m (\s{(n-2) \cdots 1}) \s{(n-2)}^{m-1} \cdots \s{1}^{m-1}\\
	&\;\;\vdots\\
	&\equiv \s{n}^m \cdots \s{1}^m,
\end{align*}
so $\s{n}^m \cdots \s{1}^m \equiv (\s{n \cdots 1})^m$.
Similarly by \eqref{3} and \eqref{11}, 
$\s{\overline 1}^m \cdots \s{\overline n}^m \equiv (\s{\overline 1 \cdots \overline n})^m$.
%\begin{align*}
%\overline{ \s 1}^{m-1} \cdots \overline{ \s n}^{m-1} \s n^{m-1} \cdots \s 1^{m-1} &\equiv \overline{ \s 1}^{m-1} \cdots \overline{ \s n}^{m-1} \s{\overline{1} \cdots \overline{n} n \cdots 1} \s n^{m-1} \cdots \s 1^{m-1} & (\ref{newidentity})  \\
%&\equiv \overline{ \s 1}^{m} \cdots \overline{ \s n}^{m} \s n^{m} \cdots \s 1^{m} & (\ref{1})(\ref{3})(\ref{10})(\ref{11}).
%\end{align*}
%Note that in the last equivalence we use (\ref{1}) and (\ref{10}) to move the factors in $\s{n \cdots 1}$ to the right and we use (\ref{3}) and (\ref{11}) to move the factors in $\s{\overline{1} \cdots \overline{n}}$ to the left. 
Then applying Lemma~\ref{newidentity} repeatedly to
\[\s{\overline 1}^m \cdots \s{\overline n}^m\s{n}^m \cdots \s{1}^m \equiv (\s{\overline 1 \cdots \overline n})^m(\s{n \cdots 1})^m\]
gives the claim.

Now suppose the proposition statement is true for all words of length less than $\ell$.  Let $x = x_1 \cdots x_{\ell}$ and $y = x_1 \cdots x_{\ell-1}$.  By induction we know the statement holds for $y$ for some $N',M' \in \NN$.  Then take $M = \max \{m(x),  M'\}$ and $N = \max \{ n(x), N' \}$ and let $m \geq M$ and $n \geq N$.   By induction we have $u_x = u_y u_{x_{\ell}} \equiv u_{[y]} u_{x_{\ell}} \pmod{J}$.  From this we see that it suffices to show $u_{[y]} u_{x_{\ell}} \equiv u_{[x]} \pmod{J}$. For ease of notation we let $\alpha_i = \alpha_i(y)$, $\beta_i = \beta^m_i(y)$, $w_i = w_i(y),$ and $\beta_i(x) = \beta_i^m(x)$ for all $i$.

We now split the argument into four cases depending on $x_\ell$ and $\alpha_i$. %In each case we first determine the $\alpha_i(x), w_i(x)$, and $\beta_i(x)$ for all $i$ since we need to know what $u_{[x]}$ looks like.  We then proceed to show the desired equivalence. 
Note that if $x_\ell = t$ or $\overline t$ for $t \geq 1$, then $\alpha_i(x) = \alpha_i$, $w_i(x) = w_i$, and $\beta_i(x) = \beta_i$ for all $i \neq t, t-1$. 

\begin{case}
Suppose $x_{\ell} = t$ and $\alpha_t = 0$.  We have $\alpha_{t-1}(x) = \alpha_{t-1} +1$, $w_{t-1}(x) = w_{t-1}$, $\alpha_t(x) = 0$, and $w_t(x) = w_t +1$, so that $ \beta_{t-1}(x) = \beta_{t-1} + 1$ and $\beta_t(x) = \beta_t + 1$.
%We now show that %$c_{[y]_{m,n}} u_{x_{\ell}} \equiv %u_{[x]_{m,n}} \pmod{J}$.  %Suppose $t =1$.  Since $t =1$ %we can ignore any of the %preceding quantities that %have $t-1$ as a subscript.  %We have
%\[
%    c_{[y]_{m,n}} c_1 \equiv % \cdots  \s 1^{\beta_1} %\overline{\s 1}^0 \s 1 \equiv % \cdots  \s 1^{\beta_1+1} %\overline{\s 1}^0  \equiv %u_{[x]_{m,n}}.
%\]
Then
\begin{alignat*}{2}
	 u_{[y]} u_t &\equiv  \cdots \s t^{\beta_t}    \s ( \s t \s - \s 1 \s)^{\beta_{ t- 1}} \overline{\s ( \s t \s - \s 1 \s)}^{\alpha_{t-1}} \cdots \s t && \\
	 &\equiv  \cdots \s t^{\beta_t}   \s ( \s t \s - \s 1 \s)^{\beta_{t-1}} \s t \overline{\s ( \s t \s - \s 1 \s)}^{\alpha_{t-1}}  \cdots  && (\ref{1}), (\ref{6})\\
	 &\equiv  \cdots \s t^{\beta_t}    \s ( \s t \s - \s 1 \s)^{\beta_{t-1}} \s t \overline{ \s t} \s t \overline{\s ( \s t \s - \s 1 \s)}^{\alpha_{t-1}}  \cdots  && (\ref{8}) \\
	 &\equiv  \cdots \s t^{\beta_t}    \s ( \s t \s - \s 1 \s)^{\beta_{t-1}} \s t \s{(t-1)} \overline{\s ( \s t \s - \s 1 \s)}^{\alpha_{t-1}+1}  \cdots \qquad && (\ref{7}) \\
	 &\equiv  \cdots \s t^{\beta_t+1}    \s ( \s t \s - \s 1 \s)^{\beta_{t-1}+1}\overline{\s ( \s t \s - \s 1 \s)}^{\alpha_{t-1}+1}  \cdots  && (\ref{2}) \\
	 &= u_{[x]}.
\end{alignat*}
\end{case}

\begin{case}
Suppose that $x_{\ell} = t$ and $\alpha_t \neq 0$.  We have $\alpha_{t-1}(x) = \alpha_{t-1} +1$, $w_{t-1}(x) = w_{t-1}$, $\alpha_t(x) = \alpha_t -1$, and $w_t(x) = w_t +1$, so that $\beta_{t-1}(x) = \beta_{t-1} + 1$ and $\beta_t(x) = \beta_t$.
%Suppose $t =1$.  We ignore %any of the preceding %quantities that have $t-1$ as %a subscript.  We have
%\begin{align*}
%    c_{[y]_{m,n}} c_1 &\equiv % \cdots  \s 1^{\beta_1} %\overline{\s 1}^{\alpha_1} \s %1 & \\
%    &\equiv  \cdots  \s %1^{\beta_1} \overline{\s %1}^{\alpha_1 -1}  & (7) \\
%    &\equiv u_{[x]_{m,n}}.
%\end{align*}
Then
\begin{alignat*}{2}
    u_{[y]} u_t &\equiv  \cdots \s t^{\beta_t} \overline{\s t}^{\alpha_t}    \s (\s t \s - \s 1 \s )^{\beta_{ t- 1}} \overline{\s (\s t \s - \s 1 \s )}^{\alpha_{t-1}} \cdots \s t & \\
    %&\equiv  \cdots \s t^{\beta_t} \overline{\s t}^{\alpha_t}    \s (\s t \s - \s 1 \s )^{\beta_{t-1}} \s t \overline{\s (\s t \s - \s 1 \s )}^{\alpha_{t-1}}  \cdots  & (\ref{1})(\ref{6})\\
    &\equiv  \cdots \s t^{\beta_t} \overline{\s t}^{\alpha_t-1}    \s (\s t \s - \s 1 \s )^{\beta_{t-1}} \overline{\s t} \s t \overline{\s (\s t \s - \s 1 \s )}^{\alpha_{t-1}}  \cdots  && \eqref{1}, \eqref{6}\\
    %&\equiv  \cdots \s t^{\beta_t} \overline{\s t}^{\alpha_1}    \s (\s t \s - \s 1 \s )^{\beta_{t-1}} \s t \overline{ \s t} \s t \overline{\s (\s t \s - \s 1 \s )}^{\alpha_{t-1}}  \cdots  & (8) \\
    %&\equiv  \cdots \s t^{\beta_t} \overline{\s t}^{\alpha_1}    \s (\s t \s - \s 1 \s )^{\beta_{ t-1}} \s t \s (\s t \s - \s 1 \s ) \overline{ \s (\s t \s - \s 1 \s )}   \overline{\s (\s t \s - \s 1 \s )}^{\alpha_{t-1}}  \cdots  & (7) \\
    %&\equiv  \cdots \s t^{\beta_t} \overline{\s t}^{\alpha_1}  \s t   \s (\s t \s - \s 1 \s )^{\beta_{ t-1}} \s (\s t \s - \s 1 \s ) \overline{ \s (\s t \s - \s 1 \s )}   \overline{\s (\s t \s - \s 1 \s )}^{\alpha_{t-1}}  \cdots  & (2) \\
    %&\equiv  \cdots \s t^{\beta_t} \overline{\s t}^{\alpha_1-1} \overline{ \s t} \s t  \s (\s t \s - \s 1 \s )   \s (\s t \s - \s 1 \s )^{\beta_{ t-1} }     \overline{\s (\s t \s - \s 1 \s )}^{\alpha_{t-1} + 1}  \cdots  &  \\
    %&\equiv  \cdots \s t^{\beta_t } \overline{\s t}^{\alpha_1 - 1} \s (\s t \s - \s 1 \s ) \overline{\s (\s t \s - \s 1 \s )} \s (\s t \s - \s 1 \s )  \s (\s t \s - \s 1 \s )^{\beta_{ t-1} }     \overline{\s (\s t \s - \s 1 \s )}^{\alpha_{t-1} + 1}  \cdots  &(7)\\
    %&\equiv  \cdots \s t^{\beta_t } \overline{\s t}^{\alpha_t - 1}   \s (\s t \s - \s 1 \s )^{\beta_{ t-1}} \s{(t-1) \overline{(t-1)}}  \,   \overline{\s (\s t \s - \s 1 \s )}^{\alpha_{t-1} }  \cdots  & (\ref{7}) \\
    &\equiv  \cdots \s t^{\beta_t } \overline{\s t}^{\alpha_t - 1}   \s (\s t \s - \s 1 \s )^{\beta_{ t-1} + 1}     \overline{\s (\s t \s - \s 1 \s )}^{\alpha_{t-1} + 1}  \cdots \qquad && \eqref{7} \\
    &= u_{[x]}.
\end{alignat*}

\end{case}

\begin{case}
Suppose that $x_{\ell} = \overline{t}$ and $\alpha_{t-1} = 0$.  We have $\alpha_{t-1}(x) = 0$, $w_{t-1}(x) = w_{t-1}$, $\alpha_t(x) = \alpha_t + 1$, and $w_t(x) = w_t - 1$, so that $\beta_{t-1}(x) = \beta_{t-1}$ and $\beta_t(x) = \beta_t$.
%Suppose $t = 1$.  Ignoring %any of the preceding %quantities that have $t-1$ as %a subscript, we have
%\[
%c_{[y]_{m,n}} %c_{\overline{1}} \equiv  %\cdots  \s 1^{\beta_1} %\overline{\s 1}^{\alpha_1} %\overline{ \s 1} \equiv  %\cdots  \s 1^{\beta_1} %\overline{\s 1}^{\alpha_1 + %1} \equiv u_{[x]_{m,n}}.
%\]
Then
\begin{alignat*}{2}
    u_{[y]} d_t &\equiv  \cdots \s t^{\beta_t} \overline{\s t}^{\alpha_t}   \s (\s t \s - \s 1 \s )^{\beta_{ t- 1}} \cdots \overline{ \s t} & \\
    %&\equiv  \cdots \s t^{\beta_t} \overline{\s t}^{\alpha_t} \overline{ \s t}   \s (\s t \s - \s 1 \s )^{\beta_{ t- 1}} \cdots  & (\ref{3})(\ref{6})\\
    &\equiv  \cdots \s t^{\beta_t} \overline{\s t}^{\alpha_t + 1}   \s (\s t \s - \s 1 \s )^{\beta_{ t- 1}} \cdots  \qquad&& \eqref{3}, \eqref{6}\\
    &= u_{[x]}.
\end{alignat*}
\end{case}

\begin{case}
Finally, suppose that $x_{\ell} = \overline{t}$ and $\alpha_{t-1} \neq 0$. We have $\alpha_{t-1}(x) = \alpha_{t-1} - 1$, $w_{t-1}(x) = w_{t-1}$, $\alpha_t(x) = \alpha_t + 1$, and $w_t(x) = w_t - 1$, so that $\beta_{t-1}(x) = \beta_{t-1} - 1$ and $\beta_t(x) = \beta_t$.  Then
%Suppose $t = 1$.  Ignoring %any of the preceding %quantities that have $t-1$ as %a subscript, we have
%\[
%c_{[y]_{m,n}} %c_{\overline{1}} \equiv  %\cdots  \s 1^{\beta_1} %\overline{\s 1}^{\alpha_1} %\overline{ \s 1} \equiv  %\cdots  \s 1^{\beta_1} %\overline{\s 1}^{\alpha_1 + %1} \equiv u_{[x]_{m,n}}.
%\]
\begin{alignat*}{2}
    u_{[y]} d_t &\equiv  \cdots \s t^{\beta_t} \overline{\s t}^{\alpha_t}   \s (\s t \s - \s 1 \s )^{\beta_{ t- 1}} \overline{\s (\s t \s - \s 1 \s )}^{\alpha_{t-1}} \cdots \overline{ \s t} & \\
    &\equiv  \cdots \s t^{\beta_t} \overline{\s t}^{\alpha_t}   \s (\s t \s - \s 1 \s )^{\beta_{ t- 1}} \overline{\s (\s t \s - \s 1 \s )}^{\alpha_{t-1}} \overline{ \s t} \cdots  && \eqref{3}, \eqref{6}\\
    &\equiv  \cdots \s t^{\beta_t} \overline{\s t}^{\alpha_t}    \s (\s t \s - \s 1 \s )^{\beta_{ t- 1}-1} \s (\s t \s - \s 1 \s ) \overline{\s (\s t \s - \s 1 \s )} \overline{ \s t} \overline{\s (\s t \s - \s 1 \s )}^{\alpha_{t-1}-1}  \cdots  \qquad&& (\ref{4})\\
    &\equiv  \cdots \s t^{\beta_t} \overline{\s t}^{\alpha_t}   \s (\s t \s - \s 1 \s )^{\beta_{ t- 1}-1} \overline{ \s t} \s t \overline{ \s t} \overline{\s (\s t \s - \s 1 \s )}^{\alpha_{t-1}-1}  \cdots  && (\ref{7})\\
    &\equiv  \cdots \s t^{\beta_t} \overline{\s t}^{\alpha_t}   \s (\s t \s - \s 1 \s )^{\beta_{ t- 1}-1}  \overline{ \s t} \overline{\s (\s t \s - \s 1 \s )}^{\alpha_{t-1}-1}  \cdots  && (\ref{9})\\
    %&\equiv  \cdots \s t^{\beta_t} \overline{\s t}^{\alpha_t}  \overline{ \s t} \s (\s t \s - \s 1 \s )^{\beta_{ t- 1}-1}   \overline{\s (\s t \s - \s 1 \s )}^{\alpha_{t-1}-1}  \cdots  & (\ref{6})\\
    &\equiv  \cdots \s t^{\beta_t} \overline{\s t}^{\alpha_t+1}  \s (\s t \s - \s 1 \s )^{\beta_{ t- 1}-1}   \overline{\s (\s t \s - \s 1 \s )}^{\alpha_{t-1}-1}  \cdots  && \eqref{6} \\
    &= u_{[x]}.
\end{alignat*}
\end{case}
This completes the proof.
\end{proof}

It is now easy to complete the proof of our main theorem.

\begin{theorem}
The ideals $I$ and $J$ are equal. Equivalently, the algebra of up- and down-operators for Young's lattice (generated by the $u_i$ and $d_i$) is determined by relations (1)--(5).
\end{theorem}
\begin{proof}
The inclusion $J \subseteq I$ follows from Proposition \ref{binomialideal}.  For the other direction, note that by Proposition \ref{JsubsetI} we need only prove that $u_x \equiv u_y \pmod I$ implies $u_x \equiv u_y \pmod J$ for  words $x$ and $y$.  By Proposition \ref{u_x=u_[x]}  there exist nonnegative integers $m$ and $n$ sufficiently large such that $u_x \equiv u_{[x]_{m,n}} = u_{[y]_{m,n}} \equiv u_y \pmod J$.   
\end{proof}

\section{Subalgebras}

We now turn our attention to various subalgebras generated by up- and down-operators.  We briefly discuss a subalgebra studied by the authors in \cite{LS}, and we introduce two other subalgebras of interest, giving a complete list of relations for each of them.

\subsection{Up-operators and down-operators}  We first consider the subalgebra generated by the up-operators $u_i$. 
Let $\U'$ be the subalgebra of $\U$ generated by $u_i$ for $i \in \NN$.  Furthermore, let
$I_{\U'} = I \cap \U'$ be the ideal of $\U'$ consisting of all elements of $\U'$ that annihilate $\YY$.  We call $\U' / I_{\U'}$ the \emph{subalgebra of up-operators for Young's lattice}. 
In \cite{LS}, the present authors described this as the \emph{algebra of Schur operators} and proved the following theorem. (See also Meinel \cite{Meinel:2019dgk}.)
\begin{theorem} \label{uptheorem}
The ideal $I_{\U'}$ is generated by the following relations:
\begin{alignat*}{2}
    u_{i} u_{j} &\equiv u_{j}  u_{i} \qquad&& \text{for }|i - j|  \geq 2,  \\
    u_i  u_{i +  1} u_i &\equiv  u_{i  +  1}  u_i  u_i,  \\
    u_{i  +  1 }   u_i   u_{i  +  1}  &\equiv   u_{i  +  1}    u_{i  +  1}   u_i,  \\
    u_{i  +  1}    u_{i  +  2}   u_{i  +  1}  u_i &\equiv  u_{i  +  1}   u_{i  +  2 }   u_i u_{i  +  1}.   
\end{alignat*}
\end{theorem}
Note that most of these relations do not appear in the list of relations for the algebra of up- and down-operators, as they are implied by the quadratic relations (1)--(5) when the down-operators are included.

It is natural to also consider the subalgebra generated by the down-operators.  Let  $\D$ be the subalgebra of $\U$ generated by $d_i$ for $i \in \NN$, and let $I_{\D} = I \cap \D$.  The \emph{subalgebra of down-operators for Young's lattice} is then $\D / I_{\D}$. Recall that with respect to the basis $\YY$, we have $u_i^T = d_i$.  Applying this transpose property to the relations in Theorem~\ref{uptheorem} gives the following characterization of $\D / I_{\D}$.

\begin{theorem}
The ideal $I_{\D}$ is generated by the following relations:
\begin{alignat*}{2}
    d_{i} d_{j} &\equiv d_{j}  d_{i} \qquad &\text{for } |i - j| \geq 2, \\
    d_{i}  d_{i+1} d_{i} &\equiv  d_{i}  d_{i}  d_{i+1}, &\\
    d_{i+1}   d_{i}   d_{i+1}  &\equiv   d_{i}    d_{i+1}   d_{i+1}, &  \\
    d_{i}    d_{i+1}   d_{i  +  2}  d_{i+1} &\equiv  d_{i  +  1}   d_{i}   d_{i+2} d_{i  +  1}.   
\end{alignat*}
\end{theorem}

\subsection{$u_t$ and $d_t$ for fixed $t$}
Fix some $1 < t \in \NN$.  Let $\B$ be the subalgebra of $\U$ generated by $u_t$ and $d_t$, and consider the subalgebra $\B/I_\B = \B/(I_\U \cap \B)  \subseteq \U/I_\U$. We will show that its ideal of relations $I_{\B}$ is generated by
\begin{align}
    u_t^{i+1} d_t^{i}  &\equiv  u_t^{i+1} d_t^{i+1}  u_t  \label{21} \\
    u_t^{i}  d_t^{i+1} &\equiv  d_t  u_t^{i+1}  d_t^{i+1} \label{22}
\end{align}
for all $i \in \NN$.  Let $J_{\B}$ be the ideal generated by relations (\ref{21}) and (\ref{22}), so that we wish to show $J_\B=I_\B$.

(When $t=1$, it is straightforward to verify that the only relation between $u_1$ and $d_1$ is \eqref{5}, namely $d_1u_1\equiv id$, as this relation can be used to rewrite any monomial in the form $u_1^id_1^j$, and all such monomials act independently on $\YY$.)

\subsubsection{Peaks and valleys}
One convenient way to interpret a word consisting only of the letters $t$ and $ \overline{t}$ is as a graph of diagonal steps.  More precisely, we construct a graph corresponding to a word $x$ in the following way. Starting at the origin in the plane we read $x$ from right to left.  When we encounter a $t$ we take a diagonal step up and to the left by adding $(-1,1)$, and when we encounter a $\overline{t}$ we take a diagonal step down and to the left by adding $(-1,-1)$. One must be careful since we are reading both the word and its graph from right to left. 

We call a point of the graph with maximal height a \emph{peak} and a point with minimal height a \emph{valley}. (Peaks and valleys need not be unique.) It is straightforward to see that if $(a,b)$ is a peak and $(c,d)$ is a valley, then $\alpha_{t-1}(x) = b$ and $\alpha_t(x) = -d$. Also note that if $(e,f)$ is the (leftmost) endpoint of the graph, then $w_t(x) =f$. Therefore by Corollary~\ref{whensame}, the action of $x$ on $\YY$ is determined entirely by the heights of its peaks, valleys, and endpoint.
\begin{example} \label{graphexample}
The word $x = t^2 \overline{t}^4 t^3$ has the graph shown below.
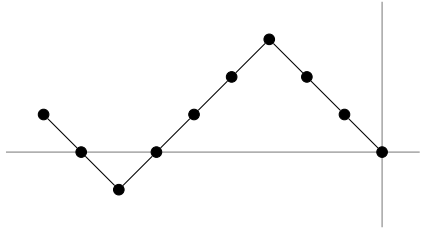
\[
\begin{tikzpicture}[scale=.5]
\draw[gray] (0, -2) -- (0, 4);
\draw[gray] (1,0) -- (-10,0);
\tikzstyle{w}=[circle, draw, fill=black, inner sep=0pt, minimum width=4pt]
\node[w] (a) at (0,0) {};
\node[w] (b) at (-1,1)  {};
\node[w] (c) at (-2,2)  {};
\node[w] (d) at (-3,3)  {};
\node[w] (e) at (-4,2)  {};
\node[w] (f) at (-5,1) {};
\node[w] (g) at (-6,0)  {};
\node[w] (h) at (-7, -1) {};
\node[w] (i) at (-8,0)  {};
\node[w] (j) at (-9,1)  {};
\draw  (a)--(b)--(c)--(d)--(e)--(f)--(g)--(h)--(i)--(j);
\end{tikzpicture}
\]
This graph has a peak at $(-3,3)$ and a valley at $(-7,-1)$. Correspondingly, $\alpha_{t-1}(x) = 3$ and $\alpha_t(x)=1$. The leftmost point of the graph is $(-9,1)$, so $w_t(x) = 1$.
\end{example}

Note that relations (\ref{21}) and (\ref{22}) are not bounded in degree since the only condition on $i$ is that it be a nonnegative integer.  This differs from the previous algebras that we examined in that the largest degree needed in those cases was $4$ (as in the subalgebra of up-operators $\U'/I_{\U'}$).  Indeed, relations of unbounded degree are required due to the following proposition. 
\begin{proposition} \label{unbounded}
The ideal $I_{\B}$ cannot be generated by elements of bounded degree.
\end{proposition}
\begin{proof}
Suppose for contradiction that the largest degree appearing among the generators of $I_{\B}$ is $h \in \NN$.  Choose an integer $ k> h$, and let $ x= t^k$ and $y = t^k \overline{t}^k t^k$. Observe that $w(x) = w(y) = (0, \ldots, 0, k, 0, \ldots)$ and $\alpha(x) = \alpha(y) = (0, \ldots, 0, k, 0, \ldots)$, and so $u_x \equiv u_y \pmod{I_{\B}}$ by Corollary \ref{whensame}.  %Our strategy is to find a property of $x$ that is invariant modulo $I_{\B}$, that is, if $u_x \equiv u_{x'} \pmod{I_{\B}}$ then $x'$ has this property too.  In particular, the invariant nature of the property will be implied by our earlier assumption about degree boundedness.  We then show that $y$ does not have this property and so it cannot be that $u_x \equiv u_y \pmod{I_{\B}}$, giving a contradiction.  

Note that in the graph of $x$, there is never a peak occurring to the right of a valley.  In other words, if $x = x_1 \ldots x_k$, then there do not exist $i<j$ such that $(-i, \alpha_{t-1}(x))$ and $(-j, -\alpha_{t}(x))$ appear in the graph of $x$. We will call an instance of a peak occurring to the right of a valley a \emph{peak/valley pair}. For instance, $x$ has no peak/valley pair but $y$ does, corresponding to the suffix subwords $t^k$ and $\overline{t}^k t^k$, respectively.

%We will say that a word $z$ satisfying $\alpha_{t-1}(z) + \alpha_t(z) > h$ has the \emph{peak/valley property} if it has no peak/valley pair. Thus $x$ has the peak/valley property, but $y$ does not since it has a peak at $(-k,k)$ and a valley at $(-2k,0)$.
%More formally, there does not exist a %suffix subword $\tilde{x} = x^{(2)} %x^{(1)}$ such that %$\alpha_{t-1}(x^{(1)}) = %\alpha_{t-1}(x) %= k$ and %$\alpha_t(\tilde{x}) = \alpha_t(x) = %0$.For any word, the absence of any %suffix subword having the properties of %$\tilde{x}$ shall be our invariant %property. 

We now show that for words $z$ satisfying $\alpha_{t-1}(z)+\alpha_t(z) > h$, our degree boundedness assumption implies that the existence of a peak/valley pair is invariant modulo $I_{\B}$. This will then lead to an immediate contradiction when applied to $x$ and $y$. Let $u_m-u_{m'}$ be a generator of $I_\B$ of degree at most $h$. It suffices to show that if the word $z=m_1mm_2$ has a peak/valley pair, then so does $z'=m_1m'm_2$.

Since $u_m \equiv u_{m'} \pmod {I_\B}$, the graphs of $m$ and $m'$ must have their peaks, valleys, and endpoints at the same heights. Therefore $z$ has a peak or valley within $m$ if and only if $z'$ has a peak or valley within $m'$. If $z$ has a peak/valley pair with neither peak nor valley occurring within $m$, then $z'$ has a peak/valley pair at the same locations. If at most one of the peak or valley occurs within $m$, say the peak, then the valley must occur within $m_1$, so $z'$ will have a peak within $m'$ and a valley within $m_1$ and hence a peak/valley pair. (The other case is similar.) The only remaining possibility is if both the peak and valley occur within $m$ (for they might switch order in $m'$). However, since $\alpha_{t-1}(z) + \alpha_t(z) > h$, the difference in height between the peak and valley is more than $h$, so they cannot both appear within $m$, which has length at most $h$. This completes the proof. 
\end{proof}

\subsubsection{Proof of relations}
We now prove that relations \eqref{21} and \eqref{22} suffice. The proofs for the following two propositions are essentially the same as the proofs of the analogous propositions in Section \ref{secupanddown}.

\begin{proposition} \label{tbinomialideal} \label{binomialideal_updown}
	The ideal $I_{\B}$ is a binomial ideal.
\end{proposition}

\begin{proposition} \label{J_BinI_B}
	The inclusion of ideals $J_\B \subseteq I_\B$ holds.
\end{proposition}

As in Section \ref{secupanddown}, our approach is to construct a standard equivalence class representative $u_{[x]}$ (modulo $I_{\B}$) for every monomial $u_x$ and to then show that $u_x \equiv u_{[x]} \pmod{ J_{\B}}$. 

\begin{definition}
For any word $x$ in $t$ and $\overline{t}$, define
\[[x] = t^{w_t(x) + \alpha_t(x)} \overline{t}^{\alpha_{t-1}(x) + \alpha_t(x)} t^{\alpha_{t-1}(x)}.\]
We say that such a word $[x]$ is the \emph{standard representative} for $x$, or alternatively that it is in \emph{standard form}.
%We call $u_{[x]}$ the standard equivalence class representative of $u_x$ modulo $I_{\B}$.  
\end{definition}
Note that all the exponents appearing in $[x]$ are nonnegative: in particular, by the definition of $\alpha_t(x)$ we have $\alpha_t(x) \geq -w_t(x)$, and so $w_t(x) + \alpha_t(x) \geq 0$. It is straightforward to check that $w(x) = (0, \ldots,0, w_t(x), 0, \ldots) = w([x])$ and $\alpha(x) = (0, \ldots,0, \alpha_{t-1}(x), \alpha_t(x), 0, \ldots) = \alpha([x])$, so Corollary \ref{whensame} implies that $[x]$ is the unique word in standard form such that $u_x \equiv u_{[x]} \pmod {I_\B}$.

\begin{proposition} \label{tequivalencerep}
Let $x = x_1 \cdots x_{\ell}$ be a word in $t$ and $\overline{t}$.  We have $u_x \equiv u_{[x]} \pmod{J_{\B}}$.
\end{proposition}
\begin{proof}
We prove this by induction on the length of $x$. If $\ell = 0$ or if $x=\overline t$, then $[x] = x$, so there is nothing to prove. If $x=t$, then $[x] = t \overline{t} t$, and $\s{t} \equiv \s{t \overline{t} t}$ by \eqref{21} for $i=0$.

Now suppose the statement holds for all words shorter than $x$.  We have that $u_x = u_{x_1}u_y$ where $y = x_2 \cdots x_{\ell}$.  By induction, $u_x=u_{x_1}u_y \equiv u_{x_1} u_{[y]} \pmod{J_{\B}}$, so we need to show $u_{x_1}u_{[y]}\equiv u_{[x]}$.

If $x_1 = t$ and $w_t(y) < \alpha_{t-1}(y)$, then $w_t(x) = w_t(y)+1$ while $\alpha(x)=\alpha(y)$. Hence $[x] = t[y]$, so there is nothing to show. Similarly if $x_1 = \overline t$ and $\alpha_t(y) = -w_t(y)$, then \[[x] = \overline{t}^{\alpha_{t-1}(y) + \alpha_t(y) + 1} t^{\alpha_{t-1}(y)} = \overline t [y],\]
so again there is nothing to show.

%In these first two cases we will see that there is no work to do since $x_1[y]$ will already be in standard form.  Suppose $x_{1} = t$ and $w_{t}(y) < \alpha_{t-1}(y)$.  This means $w_t(x) = w_t(y)+1, \alpha_{t-1}(x) = \alpha_{t-1}(y)$, and $\alpha_t(x) = \alpha_t(y)$.  We can see that $[x]$ has the same peaks and valleys as $[y]$, the only difference between the two words is their weight.  From this we see that the two words $[x]$ and $x_1[y]$ are equal.  Indeed,
%\begin{align*}
%    [x] &= t^{w_t(x) + \alpha_t(x)} \overline{t}^{\alpha_{t-1}(x) + \alpha_t(x)} t^{\alpha_{t-1}(x)} \\
%    &= t^{w_t(y) + 1 + \alpha_t(y)} \overline{t}^{\alpha_{t-1}(y) + \alpha_t(y)} t^{\alpha_{t-1}(y)} \\
%    &= x_1 [y].
%\end{align*}
%Since $x_1 [y] = [x]$ we have $u_x \equiv u_{[x]} \pmod{J_{\B}}$.  Now suppose $x_1 = \overline{t}$ and $\alpha_t(y) = -w_t(y)$.  Our relevant values are $w_t(x) = w_t(y) - 1, \alpha_{t-1}(x) = \alpha_{t-1}(y)$, and $\alpha_t(x) = \alpha_t(y) + 1$. We can see that $x_1[y]$ has the same peak as $[y]$ but that $x_1[y]$ has a valley that is one lower than the valley of $[y]$ and $w_t(x_1[y])$ is one less than $w_t([y])$.  We have
%\begin{align*}
%    [x] &= t^{w_t(x) + \alpha_t(x)} \overline{t}^{\alpha_{t-1}(x) + \alpha_t(x)} t^{\alpha_{t-1}(x)} \\
%    &= \overline{t}^{\alpha_{t-1}(y) + \alpha_t(y) + 1} t^{\alpha_{t-1}(y)} \\
%    &= x_1[y].
%\end{align*}
%Again, we are already in standard form and so $u_x \equiv u_{[x]} \pmod{J_{\B}}$.

%The next two cases will need a bit more work than the previous ones since $x_1[y]$ will not be in standard form. 
Suppose $x_1 = t$ and $w_{t}(y) = \alpha_{t-1}(y)$.  Then $w_{t}(x) = w_t(y)+1$, $\alpha_{t-1}(x) =\alpha_{t-1}(y)+ 1$, and $\alpha_t(x) = \alpha_t(y)$. Here the graph of $x$ has a new peak at its leftmost point, so $t[y]$ is not in standard form. Applying \eqref{21} with $i = w_t(y) + \alpha_t(y)$ gives
\begin{align*}
	u_t u_{[y]} &= \s{t}^{w_t(y) + \alpha_t(y) + 1} \s{\overline{t}}^{w_{t}(y) + \alpha_t(y)} \s{t}^{\alpha_{t-1}(y)}\\
	&\equiv \s{t}^{w_t(y) + \alpha_t(y) + 1} \s{\overline{t}}^{w_{t}(y) + \alpha_t(y) + 1} \s{t}^{\alpha_{t-1}(y)+1} =u_{[x]}.
\end{align*}

%Here we have that $x_1[y]$ has the same valley as $[y]$ but that its peak is one greater than the peak of $[y]$ and $w_t(x_1[y]) = w_t([y]) +1$.  Our word $x_1[y]$ is not in standard form since its only peak occurs to the left of its valley.  We must use our relations so that we have a peak to the right of the valley.  We have
%\begin{align*}
%u_{x_1}y_{[y]} &=  \s{t}^{w_t(y) + \alpha_t(y) + 1} \s{\overline{t}}^{w_{t}(y) + \alpha_t(y)} \s{t}^{\alpha_{t-1}(y)} \\
%&\equiv \s{t}^{w_t(y) + \alpha_t(y) + 1} \s{\overline{t}}^{w_{t}(y) + \alpha_t(y) + 1} \s{t} \s{t}^{\alpha_{t-1}(y)} &(\ref{21}) \\
%&= \s{t}^{w_t(x) + \alpha_t(x)} \s{\overline{t}}^{\alpha_{t-1}(x) + \alpha_t(x)} \s{t}^{\alpha_{t-1}(x)} \\
%&= u_{[x]}.
%\end{align*}
Finally, suppose $x_1 = \overline{t}$ and $\alpha_t(y) > - w_t(y)$.  We then have $w_t(x) = w_t(y)-1$ and $\alpha(x) = \alpha(y)$. % The valley and peak of $x_1[y]$ and $[y]$ are the same but $w_t(x_1[y]) = w_t([y])-1$.  
Again $\overline t[y]$ is not in standard form since it begins with $\overline{t}$.  Note that by definition $\alpha_{t-1}(y) \geq  w_t(y)$, so $\alpha_{t-1}(y)+\alpha_t(y) \geq w_t(y)+\alpha_t(y)$. Therefore we can apply \eqref{22} with $i=w_t(y)+\alpha_t(y)-1=w_t(x)+\alpha_t(x)$ to get
\begin{align*}
    d_t u_{[y]} &= \s{\overline{t}} \s{t}^{w_t(y) + \alpha_t(y)} \s{\overline{t}}^{\alpha_{t-1}(y) + \alpha_t(y)} \s{t}^{\alpha_{t-1}(y)} \\
    %&=  \s{\overline{t}} \s{t}^{w_t(y) + \alpha_t(y)} \s{\overline{t}}^{w_{t}(y) + \alpha_t(y)} \s{\overline{t}}^{\alpha_{t-1}(y) - w_t(y)} \s{t}^{\alpha_{t-1}(y)} \\
    %&=  \s{t}^{w_t(y) + \alpha_t(y) -1} \s{\overline{t}}^{w_{t}(y) + \alpha_t(y)} \s{\overline{t}}^{\alpha_{t-1}(y) - w_t(y)} \s{t}^{\alpha_{t-1}(y)} &(\ref{22})\\
    &\equiv \s{t}^{w_t(y) + \alpha_t(y)-1} \s{\overline{t}}^{\alpha_{t-1}(y) + \alpha_t(y)} \s{t}^{\alpha_{t-1}(y)} = u_{[x]}.
\end{align*}
\end{proof}

\begin{theorem}
The ideals $I_{\B}$ and $J_{\B}$ are equal.
\end{theorem}
\begin{proof}
This follows from Propositions \ref{tbinomialideal}, \ref{J_BinI_B}, and \ref{tequivalencerep}. 
\end{proof}

\subsubsection{Up- and down-operators on finite chains}

Consider again the operators $u_t$ and $d_t$ for some fixed $t > 1$. The action of these operators on $\YY$ splits up as a direct sum of the action on chains $C$, where $C$ is a set of partitions $\lambda$ that have fixed values for $\lambda'_i$ for all $i \neq t$. The action is then determined entirely by $\rho = \lambda'_{t-1} - \lambda'_{t+1}$, the difference between the $(t-1)$st and $(t+1)$st columns. (Equivalently, $C$ is a chain with $\rho+1$ elements, and $u_t$ and $d_t$ act as up- and down-operators on this chain.)

%Consider again the operators $u_t$ and $d_t$ for some fixed $t > 1$, but now also fix a positive integer $\rho$. Then $u_t$ and $d_t$ act on $\YY_\rho$, the subposet of $\YY$ consisting of $\lambda \in \YY_\rho$ such that $\lambda'_{t-1} - \lambda'_{t+1} = \rho$, that is, whose difference between their $(t-1)$st and $(t+1)$st columns equals $\rho$.  %Taking this subposet allows us to avoid the unbounded degree property from Proposition \ref{unbounded}.  
%Note that as a $\B$-module, $\YY_{\rho}$ can be viewed as the action of $\s{t, \overline{t}}$ on $\oplus_{n=1}^{\infty} c$ where $c$ is a chain such that if $\lambda \in c$, then  $\lambda'_{t-1} - \lambda'_{t+1} = \rho$.  

Fix $\rho$, and let $I_{\C}$ be the two-sided ideal of $\B$ containing all elements which annihilate $C$, a chain with $\rho+1$ elements.  We characterize the algebra $\B / I_{\C}$ by showing that $I_{\C}$ is generated by the following relations:
\begin{alignat}{2}
    u_t^{i+1} d_t^{i}  &\equiv  u_t^{i+1} d_t^{i+1}u_t  \qquad&&\text{for } 0 \leq i \leq \rho -1,   \label{new1} \\
    u_t^{i}  d_t^{i+1} &\equiv  d_t  u_t^{i+1}  d_t^{i+1} &&\text{for } 0 \leq i \leq \rho - 1, \label{new2} \\
    u_t^{\rho +1} &\equiv 0, & \label{new3} \\
    d_t^{\rho+1} &\equiv 0. & \label{new4}
\end{alignat}

Let $J_{\C}$ be the ideal generated by relations (\ref{new1})--(\ref{new4}). We will show that $J_{\C} = I_{\C}$ by exploiting the close relationship between these ideals and $I_{\B}$.
%The proof of the following proposition is similar to the analogous situation for $\U / I$.
%
%\begin{proposition} \label{tbinomialideal2} \label{binomialideal_updown}
%The ideal $I_{\C}$ is a binomial ideal.
%\end{proposition}
\begin{theorem}
The ideals $I_{\C}$ and $J_{\C}$ are equal.
\end{theorem}
\begin{proof}
Recall that $I_{\B}$ is the two-sided ideal of $\B$ containing all elements which annihilate $\YY$.  Let $P$ be the two-sided ideal of $\B$ which is generated by relations (\ref{new3}) and (\ref{new4}).  It is straightforward to see that $J_\C = I_\B+P$ (since \eqref{21} and \eqref{22} for $i \geq \rho$ are implied by \eqref{new3} and \eqref{new4}), so we need to show that $I_\C=I_\B+P$.

The inclusion $I_{\B} + P \subseteq I_{\C}$ holds since both \eqref{new3} and \eqref{new4} annihilate $C$. For the reverse direction, note that by Proposition~\ref{tequivalencerep}, $\B/I_\B$ has a basis consisting of the standard representatives $u_{[x]}$. A basis element $u_{[x]}$ annihilates $C$ if and only if the power of $\s{\overline t}$ appearing in it is larger than $\rho$, which occurs if and only if it lies in $P$. The other basis elements act independently on $C$ as in the proof of Proposition~\ref{binomialideal}. It follows that $I_{\C} \subseteq I_{\B}+P$.
%holds since \eqref{21} and \eqref{22} for $i \geq \rho$ are implied by \eqref{new3} and \eqref{new4}.  For the reverse inclusion, suppose for contradiction that there exists some $c \in I_{\C}$ such that $c \notin I_{\B} + P$.  We know by Proposition \ref{tequivalencerep} that $\B / I_{\B}$ has a basis consisting of the standard representatives $u_{[x]}$, so we can write $c \in \sum_{[x]} c_{[x]} u_{[x]} + I_{\B}$ for some $c_{[x]} \in \CC$.
%
%
%We will denote the $i$th element in this ordered basis by $u_{[x]_i}$.  Since $c \in \B / I_{\B}$ we have that $c = \sum_{i=1}^{\infty} c_i u_{[x]_i} + I_{\B}$ ($c_i \in \CC$).  This linear dependence along with our assumption that $c \notin I_{\B} + P$ implies that $c$ forces a linear dependence among the $u_{[x]_i}$ that are not contained in $P$.  Suppose such a linear dependence exists, that is, $0 = \sum a_{i} u_{[x]_{i}} + P$ where $a_i \in \CC$ and not all $a_i = 0$.  Using a method similar to the proof of Proposition \ref{binomialideal} we can fix some weight $w$ and order the $u_{[x]_i} + P$ with this weight lexicographically by their $\alpha$-vectors.  Let the lexicographically smallest of these be $u_{[x]_k} + P$.  We can construct a $\lambda \in \YY_{\rho}$ such that $(u_{[x]_k} + P)(\lambda) \neq 0$ but $(u_{[x]_i}+P)(\lambda) = 0$ for all other $u_{[x]_i}+P$ with weight $w$.  This is a contradiction and so no such linear dependence can exist and furthermore, we must have $c \in I_{\B} + P$.  
\end{proof}

\section{Conclusion}

While the results of this paper and \cite{LS} have answered various questions about up- and down-operators, there still remain directions to explore on this subject.  For instance, recall that the ideal of relations among the up- and down-operators is generated by relations of bounded degree (in fact, of degree $2$), while some subalgebras such as $\B/I_\B$ cannot be presented by relations of bounded degree.  It would be interesting to determine for which subalgebras this is true. In other words, can one characterize when the generating relations among a subset of operators are bounded versus unbounded in degree?

More generally, it would be interesting to explore these up- and down-operators for posets other than Young's lattice.  Let $P$ be a poset with an edge labeling from an index set $I$.  We can define up-operators $u_i$ for $i \in I$ such that, for $p \in P$, $u_i(p) = q$ if $p \lessdot q$ and the edge between $p$ and $q$ is labeled $i$, and otherwise $u_i(p) = 0$ if no such $q$ exists.  Note that for Young's lattice as considered above, the label between $\lambda$ and $\mu$ where $\lambda \lessdot \mu$ is the column $i$ in which the unique box of $\mu / \lambda$ appears.   One can consider the algebras generated by these operators (or the analogously defined $d_i$) and try to describe their relations for other posets of interest, such as Bruhat order or absolute order on a Coxeter group. (The case of weak order leads to the study of nil-Coxeter algebras \cite{FominStanley}.) It would also be interesting if it were possible to relate structural properties of these algebras to the structure of the corresponding posets in some way.

%Possible posets of interest are Bruhat order, absolute order, and the poset of generalized noncrossing partitions.    

\bibliographystyle{plain} 
\bibliography{references}

\end{document}